\documentclass[10pt,a4paper]{article}
\usepackage{amssymb,amsmath,amsopn,amsthm,graphicx,makeidx}
\usepackage[all,2cell]{xy} \UseAllTwocells

\begin{document}

\newtheorem{definition}{Definition}[section]
\newtheorem{definitions}[definition]{Definitions}
\newtheorem{lemma}[definition]{Lemma}
\newtheorem{prop}[definition]{Proposition}
\newtheorem{theorem}[definition]{Theorem}
\newtheorem{cor}[definition]{Corollary}
\newtheorem{cors}[definition]{Corollaries}
\theoremstyle{remark}
\newtheorem{remark}[definition]{Remark}
\theoremstyle{remark}
\newtheorem{remarks}[definition]{Remarks}
\theoremstyle{remark}
\newtheorem{notation}[definition]{Notation}
\theoremstyle{remark}
\newtheorem{example}[definition]{Example}
\theoremstyle{remark}
\newtheorem{examples}[definition]{Examples}
\theoremstyle{remark}
\newtheorem{dgram}[definition]{Diagram}
\theoremstyle{remark}
\newtheorem{fact}[definition]{Fact}
\theoremstyle{remark}
\newtheorem{illust}[definition]{Illustration}
\theoremstyle{remark}
\newtheorem{rmk}[definition]{Remark}
\theoremstyle{definition}
\newtheorem{question}[definition]{Question}
\theoremstyle{definition}
\newtheorem{conj}[definition]{Conjecture}

\renewcommand{\marginpar}[2][]{}

\newcommand{\stac}[2]{\genfrac{}{}{0pt}{}{#1}{#2}}
\newcommand{\stacc}[3]{\stac{\stac{\stac{}{#1}}{#2}}{\stac{}{#3}}}
\newcommand{\staccc}[4]{\stac{\stac{#1}{#2}}{\stac{#3}{#4}}}
\newcommand{\stacccc}[5]{\stac{\stacc{#1}{#2}{#3}}{\stac{#4}{#5}}}

\renewenvironment{proof}{\noindent {\bf{Proof.}}}{\hspace*{3mm}{$\Box$}{\vspace{9pt}}}

\title{Categories of imaginaries for definable additive categories}
\author{Mike Prest,\\ Alan Turing Building
\\Department of Mathematics\\University of Manchester\\
Manchester M13 9PL\\UK\\mprest@manchester.ac.uk}

\maketitle

\tableofcontents

\section{Introduction}\label{secintro}\marginpar{secintro}

An additive category is {\bf definable} if it is equivalent to a {\bf definable subcategory} of a module category $ {\rm Mod}\mbox{-}{\mathcal R} $, meaning a full subcategory which is closed under direct products, direct limits ($=$ directed colimits) and pure subobjects.  Such a category ${\mathcal D}$ has associated to it a canonical model theory for its objects in the sense that each object $D\in {\mathcal D}$ becomes a structure for the associated language; as such, $D$ is a model of the associated theory and ${\mathcal D}$ is the category of models for that theory.  The model theory, moreover, is essentially just like the model theory of modules over a ring, hence very amenable, in particular the theory of ${\mathcal D}$ (the common theory of the objects of ${\mathcal D}$) has pp-elimination of quantifiers.  The associated category of pp-imaginaries is abelian and every small abelian category arises, up to equivalence of categories, in this way.  In this paper we pursue some themes in the model theory of additive structures which exploit the categorical view of these categories of pp-imaginaries.

If $R$ is a ring then the class of flat left $R$-modules is definable exactly when $R$ is right coherent; this is also the condition for the class of absolutely pure (also called fp-injective) right $R$-modules to be definable.  These two, dual, results are due to Eklof and Sabbagh (\cite{EkSab}, \cite{SabEk}).

If $R$ is right coherent then the category, ${\rm mod}\mbox{-}R$, of finitely presented right $R$-modules is a (skeletally small) abelian category and plays a key role in the above-mentioned theories.  Namely, it is (equivalent to) the category of pp-imaginaries for the theory of flat left $R$-modules and its opposite, $\big({\rm mod}\mbox{-}R\big)^{\rm op}$, is the category of pp-imaginaries for the theory of absolutely pure right modules.  In Section \ref{secARflat}, we identify the skeletally small abelian category which plays the same role in the non-coherent case.  We also extend the ``classical" case by replacing the ring $R$ by a small preadditive category (otherwise known as a ring with many objects or as a ringoid).  Recall that a {\bf preadditive category} is a category where each hom set is an abelian group and where composition is bilinear (that is, the equations $ f(g+h)=fg+fh $ and $ (g+h)f=gf+hf $ hold true whenever they make sense). Such a category is {\bf skeletally small} if, up to isomorphism, it has just a set of objects; in this case we often replace the category by a {\bf small} version - one with only a set of objects but with at least one in each isomorphism class. Since such a small version is equivalent (in the category-theoretic sense) to the original one, such replacement may be made without comment since it affects nothing of substance. An example is the category, $ {\rm mod}\mbox{-}R$, of finitely presented modules over a ring $ R$: this is skeletally small but not small. In setting up a language for $ ({\rm mod}\mbox{-}R)$-modules we would want to have just a set of sorts and function symbols, so we would replace $ {\rm mod}\mbox{-}R $ by a small version of it, but the languages obtained from different choices of small versions, though literally different, are entirely equivalent in their model theory and so we never distinguish between them.

If ${\mathcal R}$ is a skeletally small preadditive category then a {\bf left} ${\mathcal R}${\bf -module} is an additive functor from ${\mathcal R}$ to the category, ${\bf Ab}$, of abelian groups (and a {\bf right} ${\mathcal R}${\bf -module} is a contravariant functor, that is, a functor from ${\mathcal R}^{\rm op}$).  If ${\mathcal R}$ has just one object or, in effect equivalently, finitely many isomorphism classes of objects, then it is essentially a ring and ``module'' means what it usually does.  Although it is fairly straightforward to extend from $R$-modules to ${\mathcal R}$-modules, this kind of thing is seldom worked out in detail, so we take the opportunity to do so here, illustrating thereby the general process.

I have tried to include enough background to make this paper accessible to algebraists and model-theorists.  I have made heavy use of \cite{PreNBK} as a convenient reference which gathers together much of the background material (\cite{PreDefAddCat} is an alternative).

Thanks to Antongiulio Fornasiero, Sonia L'Innocente and Giuseppina Terzo for pointing out a slip in Section 5, where I should have referred only to pp-, not general, imaginaries.

\section{Model theory over preadditive categories}\label{secmodRR}\marginpar{secmodRR}

First we recall some basic definitions from the model theory of modules over a ring $R$.  A homogeneous system of linear equations $ \overline{x}H=0 $ with coefficients (i.e.~entries of the matrix $H$) from $ R $ defines, in each $ R$-module $M$, a subgroup (of $M^n$, where $n$ is the length of $\overline{x}$) - the solution set of the system. If we project out some coordinates then the result is defined by a condition of the form $ \exists \overline{y} \,\, \overline{x}\overline{y}H=0 $ (different $ \overline{x}$ and where $\overline{x}\overline{y}$ is the row vector consisting of the entries of $\overline{x}$ followed by those of $\overline{y}$) - that's a {\bf pp formula}, typical notation for which is $ \phi  $ or, if we want to show the {\bf free}(=unquantified) {\bf variables}, $ \phi (\overline{x})$.  Again, the solution set in any module $M$ - we write $\phi(M)$ for that - is a subgroup of the appropriate power of $M$.  Write $ \psi \leq \phi  $ if $ \psi (M)\leq \phi (M) $ for every module $ M$ (checking for $M$ finitely presented is enough, see, e.g.~\cite[1.2.23]{PreNBK}); such a {\bf pp-pair} is said to be {\bf closed on} a module $ M $ if $ \psi (M)=\phi (M)$, otherwise it is {\bf open on} $ M$. Given a set, $ \{ \psi _\lambda \leq \phi _\lambda \}_\lambda  $, of pp-pairs, the class of modules $ M $ on which each of these pairs is closed is a {\bf definable subclass} of $ {\rm Mod}\mbox{-}R$ (and the full subcategory on these is a typical definable subcategory); note that the modules in this class are exactly those which satisfy the axioms $ \forall \overline{x} \big( \phi _\lambda (\overline{x})\rightarrow \psi _\lambda (\overline{x})\big) $ - a set of coherent (indeed regular) sentences in the first order theory of $ R$-modules.  These are exactly the axiomatisable subclasses which are closed under direct sums and direct summands.  Recall that these classes of modules are in bijection with the closed sets of the Ziegler spectrum of $ R$ and, if we allow a skeletally small preadditive category ${\mathcal R}$ in place of $R$, we obtain arbitrary definable additive categories (for background we refer to, for instance, \cite{Zie}, \cite{PreBk}, \cite{PreNBK}).

From the point of model theory the key role of the pp formulas among the others is explained by the fact that every formula reduces to a boolean combination of them; equivalently, every definable set is a finite boolean combination of solution sets of pp formulas.  We write $\chi(M)$ for the solution set in a module $M$ of any formula $\chi$; if $\psi\leq \phi$ then we write $(\phi/\psi)(M)$ for the quotient group $\phi(M)/\psi(M)$.

\begin{theorem}\label{ppeq}\marginpar{ppeq} (pp-elimination of quantifiers) (see \cite[p.~36]{PreBk} for references) Let $ R $ be any ring.

\noindent (1) If $ \sigma  $ is a sentence in the language of $ R$-modules then
there is a
finite boolean combination, $ \tau  $, of sentences of the form $ {\rm card}(\phi(-)
/\psi(-) )\geq m, $ where $ \phi \geq \psi  $ are pp conditions and $ m $ is a positive integer,
such that $ \sigma  $ is equivalent to $ \tau  $ in the sense that for every
$ R$-module $M$, $ \sigma  $ is true in $M$ iff $ \tau $ is true in $M$.

\noindent (2) If $ \chi $ is any formula in the
language of $ R$-modules then there is a sentence, $ \tau  $, and a finite boolean
combination, $ \eta $, of pp formulas such that for every
module $ M $ and tuple $
\overline{a} $ from $ M $ (matching the free variables of $ \chi$) we have $
 \overline{a} \in \chi (M) $ iff both $ \tau  $ is true in $M$ and $ \overline{a}\in \eta
(M). $ In particular, the solution set to $ \chi $ in every module $ M
$ is a finite boolean combination (i.e.~using $\cap$, $\cup$, $^c$) of pp-definable subgroups.  If non-zero constants are allowed in $\chi$ then the solution set will be a finite boolean combination of cosets of pp-definable subgroups.
\end{theorem}

\vspace{4pt}

Extending to modules over rings, ${\mathcal R}$, with more than one object entails that modules have to be treated as having more than one sort of element:  for each object $P$ in the category ${\mathcal R}$ we introduce a sort. (In fact all we need is at least one object $P$ from each isomorphism type of object in ${\mathcal R}$ - so we can assume that ${\mathcal R}$ is small.)  There are two ways to describe the elements of a given sort of a module $M$, depending on how we choose to view modules.

Literally, a right module over ${\mathcal R}$ is a contravariant additive functor, that is, a functor from ${\mathcal R}^{\rm op}$, to ${\bf Ab}$.  So a right ${\mathcal R}$-module $M$ assigns to each object $P$ of ${\mathcal R}$ an abelian group $MP$ - that will be the set of elements of $M$ of sort $P$ - and assigns to each arrow $r:P\rightarrow Q$ in ${\mathcal R}$ a homomorphism $Mr:MQ\rightarrow MP$ of abelian groups.  The language has, for each sort, a symbol for addition in and for the zero object of that sort (so as to capture the abelian group structure of $MP$) and has, for each arrow $r$, a unary function symbol (which, for simplicity, we also denote $r$) from sort $MQ$ to sort $MP$ - the interpretation of which is, of course, the function $Mr$.  In the case of left = covariant modules the function symbol corresponding to the arrow $r$ would go from the sort corresponding to $P$ to the sort corresponding to $Q$.

If ${\mathcal R}$ is actually a ring, that is has just one object, $\ast$ say, so all the arrows go from $\ast$ to $\ast$, then we can write $R=(\ast,\ast)$ for the endomorphism ring of that object.  A module $M$ then consists of a single abelian group $M\ast$ together with, for each $r\in R$, an endomorphism, multiplication-by-$r$.  One may check that the conditions for being a (contravariant) functor are exactly those for this data to describe a (right) module.

In the general case, if we think of the image of the functor as a structure then it is a ``multi-sorted module", meaning that for each sort (fix $P$) we have a module over some specified ring (${\rm End}(P)$) and then we have a specified pattern of abelian group homomorphisms (the various $r$) between these.

An alternative view of an $R$-module is based on the standard isomorphism $M_R\simeq (R_R,M_R)_R$ - the isomorphism between a right $R$-module $M_R$ and the group of $R$-module homomorphisms from the free module, $R_R$, to $M_R$ (this group of homomorphisms has a natural right $R$-module structure induced by the left action of $R$ on itself).  Thus the elements of $M$ of the unique sort are homomorphisms with a given domain ($R_R$).  To see how this generalises, one has to note that the projective module $R_R$ is the representable functor $(-, \ast)$ on the ring viewed as a 1-object category.  So, for a general ${\mathcal R}$ we replace each object $P$ of ${\mathcal R}$ by the corresponding representable functor $(-,P)$ - this is a projective right module ${\mathcal R}$-module - and we regard the elements of $M$ of sort $P$ as the elements of the group $((-,P),M)$ of ${\mathcal R}$-module homomorphisms.  That $((-,P),M)$ is naturally isomorphic to $MP$ is exactly the Yoneda Lemma, so the two views are equivalent.

To be clear about sorting of variables:  each variable $x$ of the language that we use has a specified sort, with variables of sort (corresponding to) $P$ being replaceable by elements of that sort; solution sets of formulas will now be subgroups of products of the form $MP_1 \times \dots \times MP_n$ where $P_1,\dots, P_n$ are objects of ${\mathcal R}$ (that product will be isomorphic to $M(P_1\oplus \dots \oplus P_n)$ if the direct sum $P_1\oplus \dots \oplus P_n$ happens to exist in ${\mathcal R}$).  For a simple example of a pp formula with free variable $x_P$ of sort $P$ we have $\exists y_Q (x_P =y_Qr)$ if $r:P\rightarrow Q$ is an arrow in ${\mathcal R}$; the solution set in a right module $M$ is $Mr$, meaning the image of the arrow $Mr:MQ\rightarrow MP$.  Everything goes as in the ring/one-object case with only minor modifications (more details can be found at \cite[\S10.2.4, Appx.~B]{PreNBK} and there will be further illustrations here).

\section{Flat and absolutely pure modules}\label{secflat}\marginpar{secflat}

Suppose that $R$ is a ring.  We say that a right $R$-module $M$ is {\bf flat} if it satisfies the following equivalent conditions (which may be found in many sources, for instance \cite[\S\S 2.1, 2.3]{PreNBK}).

\begin{theorem}\label{charflat}\marginpar{charflat} The following conditions on a module $M_R$ are equivalent.

\noindent (i)  If $f:A\rightarrow B$ is a monomorphism between (finitely presented) left $R$-modules then the morphism $1_M\otimes f: M\otimes A \rightarrow M\otimes B$ is a monomorphism of abelian groups, where $\otimes$ means $\otimes _R$.

\noindent (ii)  The functor $(M,-): {\rm mod}\mbox{-}R \rightarrow {\bf Ab}$ is exact.

\noindent (iii)  The module $M$ is a direct limit of projective right $R$-modules.
\end{theorem}

The dual (in various senses) notion is that of an {\bf absolutely pure} module, which is as follows.

\begin{theorem}\label{charabs}\marginpar{charabs}   The following conditions on a module $M_R$ are equivalent.

\noindent (i) Every monomorphism of $R$-modules with domain $M$ is pure (for which see just below).

\noindent (ii) The module $M$ is a pure submodule of its injective hull.

\noindent (iii) For every monomorphism $j:C\rightarrow D$ of right $R$-modules with finitely presented cokernel, every morphism $C\rightarrow M$ factors through $j$.

\noindent (iv)  The functor ${\rm Ext}^1_R(-,M): {\rm mod}\mbox{-}R \rightarrow {\bf Ab}$ is exact.
\end{theorem}

We say that a monomorphism is {\bf pure} if it satisfies the following equivalent conditions.

\begin{theorem}\label{charpure}\marginpar{charpure} The following conditions on a monomorphism $j:M\rightarrow N$ of right $R$-modules are equivalent.

\noindent (i)  For every (finitely presented) left $R$-module $L$ the morphism $j\otimes 1_L: M\otimes L \rightarrow N\otimes L$ is monic.

\noindent (ii)  The morphism $j$ is a direct limit of split monomorphisms.

\noindent (iii)  For every pp formula $\phi(x_1,\dots,x_n)$ we have $\phi(M)=M^n\cap \phi(N)$ (on identifying $M$ with its image in $N$).

\noindent (iv)  As (iii) but for pp formulas with one free variable.
\end{theorem}

The terminology is extended by saying that a short exact sequence $0\rightarrow A \xrightarrow{j} B \xrightarrow{p} C \rightarrow 0$ is {\bf pure-exact} if $j$ is a pure monomorphism, in which case we also say that $p$ is a {\bf pure epimorphism}; of these there are again various characterisations including the following.

\begin{theorem} (\cite{Rot1pure})  An morphism $p:B\rightarrow C$ is a pure epimorphism iff for every pp formula $\phi(\overline{x})$ and $\overline{c}$ from $C$ there is $\overline{b}\in \phi(B)$ with $p\overline{b} =\overline{c}$.
\end{theorem}

There are characterisations of flat and absolutely pure modules in terms of their pp-definable subgroups.  For the former the result is due to Zimmermann (and independently by Rothmaler), see \cite[2.3.9]{PreNBK}.

\begin{theorem}\label{flatchar}\marginpar{flatchar} (\cite[1.3]{Zimm1}, \cite[Prop.~4]{Rotflat})  A module $M$ is flat iff for every pp formula $\phi(x_1,\dots,x_n)$ we have $\phi(M)=M.\phi(R)$ where the latter means $\{ \sum_{i=1}^k \overline{a}_i. \overline{r}_i : \overline{a}_i\in M^n, \overline{r}_i\in \phi(R)\mbox{ some }k\}$ and where $\overline{a}. \overline{r}$ means the $n$-tuple whose $j$th entry is the $j$th entry of $\overline{a}$ times the $j$th entry of $\overline{r}$.  In fact, it is enough to check for pp formulas with one free variable.
\end{theorem}

\begin{theorem}\label{abspurchar}\marginpar{abspurchar} \cite[1.3]{PRZ1} A module $M$ is absolutely pure iff for every pp formula $\phi$ we have $\phi(M) = {\rm ann}_MD\phi(_RR)$, where $D$ denotes the duality for pp formulas described in Section \ref{secdual} (and ``${\rm ann}_M(-)$" denotes the annihilator in $M$ of $(-)$).
\end{theorem}

\subsection{Extending to preadditive categories; coherence}\label{secflatRR}\marginpar{secflatRR}

An important ingredient of the proofs of the model-theoretic results above is the fact that if $\phi(x)$ is a pp formula then, applied to the right module $R_R$, we obtain a left ideal:  $\phi(R_R)$ is a left ideal because ``pp formulas are preserved by homomorphisms":  if $f:M\rightarrow N$ is a morphism of modules and $\phi$ is pp then $f\phi(M)\leq \phi(N)$; in particular, they are preserved by endomorphisms.  In the case of more general preadditive categories, ${\mathcal R}$, in place of $R$ there is a notion of left ideal but it is not the case that ${\mathcal R}$ is a module over itself (that is, ``${\mathcal R}_{\mathcal R}$" does not make obvious sense).  Nevertheless, we can obtain a left ideal by applying a pp formula for right modules to ${\mathcal R}$, as follows.

We say that a collection of morphisms of the skeletally small preadditive category ${\mathcal R}$ is a {\bf left} (respectively {\bf right}) {\bf ideal} if it is nonempty and is closed under addition and under post- (resp.~pre-) composition with any morphism in ${\mathcal R}$.  (This is a rather more general definition than that of \cite{Mit}, which is being a subfunctor of a representable functor; that definition of a left ideal would imply that all morphisms in it share the same domain.)  We say that a left ideal is {\bf finitely generated} if it contains finitely many morphisms which generate, under addition and left composition the ideal.

If $P$ is an object of ${\mathcal R}$ then by $_{\mathcal R}P$, respectively $P_{\mathcal R}$, we will denote the left, resp.~right, ${\mathcal R}$-module $(P,-)$, resp.~$(-,P)$.  Note that this is, at least can be thought of, as a left (resp.~right) ideal of ${\mathcal R}$, namely as the set, $\bigcup_{Q\in {\mathcal R}} (P,Q)$, of its elements (the union over all sorts).

If $\phi$ is a pp formula for right ${\mathcal R}$-modules, with one free variable, of sort $P$ say, then we will write $\phi({\mathcal R})$ for the left ideal defined thus:  for each object $Q$ of ${\mathcal R}$ we take the set $\phi(Q_{\mathcal R})$ - the group of those elements of $Q_{\mathcal R} =(-,Q)$ of sort $P$ which satisfy the condition $\phi$; this is a subgroup of $((-,P), (-,Q)) = (P,Q)$.  Since pp formulas are preserved by morphisms, this, that is $\bigcup_{Q\in {\mathcal R}} \phi(Q)$, is a left ideal, consisting of a collection of morphisms all with domain $P$; thus it is a sub-ideal of $_{\mathcal R}P = (P,-)$. For instance, if $\phi$ is the formula $x_P=x_P$ then the left ideal it defines is the collection of all morphisms with domain $P$.  As seen at the end of Section \ref{secmodRR}, if $r:P\rightarrow Q$ is an arrow in ${\mathcal R}$ then $\exists y \,(x=yr)$ defines the image of the map $(r,-):(Q,-) \rightarrow (P,-)$ between left ideals - a left ideal contained in $_{\mathcal R}P$.  Dually the formula $yr=0$ defines the subideal of $(Q,-)$ consisting of all those arrows in ${\mathcal R}$ with domain $Q$ which, precomposed with $r$, give $0$.

Note that the left ideal generated by (i.e.~minimal containing) a set of morphisms consists of those which can be written in the form $\sum_1^n t_ir_i$ with the $r_i$ in the generating set; this expression makes sense only if the $r_i$ share the same domain and the $t_i$ share the same codomain.  If, however, ${\mathcal R}$ is additive, so has finite direct sums, then we can make sense of such an expression, without restriction, by forming direct sums over the various domains and codomains and identifying morphisms with corresponding component morphisms between direct sums.

We say that a skeletally small preadditive category ${\mathcal R}$ is {\bf left coherent} if for each object $P$ of ${\mathcal R}$ the intersection of any two finitely generated subideals of the left ideal $(P,-)$ is again finitely generated and, if for any $s:P'\rightarrow P$ the {\bf left annihilator} of $s$, that is the set of morphisms with domain $P$ which, precomposed with $s$, give $0$, also is finitely generated.  Similarly we define right coherence.  The characterisation of coherent rings in terms of their categories of modules generalises to the following result.  An object in an additive category is {\bf coherent} if it is finitely presented and if every finitely generated subobject is finitely presented.

\begin{theorem}\label{cohringoid}\marginpar{cohringoid} (\cite[4.1]{ObRoh}) Let ${\mathcal R}$ be a skeletally small preadditive category.  Then ${\mathcal R}$ is right coherent iff the category ${\rm Mod}\mbox{-}{\mathcal R}$ of right ${\mathcal R}$-modules is locally coherent, in the sense that every finitely presented module is coherent.
\end{theorem}

In the case that ${\mathcal R}$ is a ring it is clear that a finitely generated left ideal is pp-definable:  if $r_1, \dots, r_n$ are elements of $R$ then the left ideal they generate is the solution set to the pp formula $\exists y_1,\dots, y_n \,(x=\sum_1^n y_ir_i)$.  In the general case, however, the corresponding assertion is that if a left ideal is finitely generated then it is defined by a {\bf positive existential} ({\bf pe}) formula - a disjunction of pp formulas.  For example if $r_1:P_1\rightarrow Q_1$ and $r_2:P_2\rightarrow Q_2$ are arrows in ${\mathcal R}$ and if $P_1 \neq P_2$ then the left ideal ``${\mathcal R}r_1 +{\mathcal R}r_2$" generated by $r_1, r_2$ is defined by $\exists y_1 (x_1=y_1r_1) \vee \exists y_2 (x_2=y_2r_2)$, where the variable $x_i$ has sort $P_i$.  On the other hand, if $P_1=P_2$ then that formula does not define the left ideal, rather, one needs $\exists y_1, y_2 (x=y_1r_1+y_2r_2)$ where $x$ has sort $P_1=P_2$.  Having to distinguish cases like this would be unsatisfactory.  One remedy is to adopt the narrower definition of left ideal as being a subfunctor of a representable functor $(P,-)$.  Another, which is just the natural extension of the use of $n$-tuples of variables, is to assume that ${\mathcal R}$ is {\bf additive} (i.e.~has finite direct sums) - that does not change the category of modules (up to equivalence) - and to close all ideals under both taking components and forming matrices, meaning that if $\{r_{ij}:P_i\rightarrow Q_j \}_{ij}$ is a set of morphisms in a left ideal then the morphism $r=(r_{ij})_{ij}:\bigoplus_i P_i \rightarrow \bigoplus_j Q_j$ also is in the left ideal and, conversely, if such a morphism $r$ is in a left ideal then so are all its components $r_{ij}$.  Both are reasonable solutions, each allowing one to say, correctly, that a finitely generated left ideal is definable by a pp formula for right modules (the second using a pp formula with single free variable of the product sort).  We will use the first since it is simpler, but note that what we do also covers the more general (and perhaps more useful) notion of ideal allowed in the second formulation.

Having, we hope, said what is needed to clarify these issues, the generalisations of results, such as the next, from the ${\mathcal R} =R$ case becomes rather automatic.

\begin{theorem}\label{cohfgpp}\marginpar{cohfgpp} The skeletally small preadditive category ${\mathcal R}$ is right coherent iff for every pp formula $\phi(x)$ for right ${\mathcal R}$-modules, with $x$ of sort $P\in {\mathcal R}$, the left ideal $\phi({\mathcal R}_{\mathcal R})$ is a finitely generated left ideal (contained in $(P,-)$).

If ${\mathcal R}$ is additive then equivalent is:  for every pp formula $\phi(x_1,\dots,x_n)$ for right ${\mathcal R}$-modules, with $x_i$ of sort $P_i$, the left ideal $\phi({\mathcal R}_{\mathcal R})$ is a finitely generated left ideal (contained in $(P_1\oplus \dots \oplus P_n,-)$).
\end{theorem}
\begin{proof} We use the proof given at \cite[2.3.19]{PreNBK} for the ring version (\cite[1.3]{Zimm1}, \cite[Prop.~7]{Rotflat}).

Assume that ${\mathcal R}$ is right coherent and also additive and suppose that $\phi(x_1,\dots,x_n)$ is $\exists x_{n+1},\dots, x_k ( \overline{x}H=0)$ where $x_i$ has sort $P_i$ and where $H$ is a $k\times m$ matrix with $ij$-entry $r_{ij}:P'_j\rightarrow P_i$, equivalently $(r_{ij},-):(P_i,-)\rightarrow (P'_j,-)$.  The solution set to this system of equations is the kernel of the morphism from $(P_1\oplus \dots \oplus P_k,-)$ to $(P'_1\oplus \dots P'_m,-)$ defined by $\overline{z} \mapsto \overline{z}H$.  If we write the columns of $H$ as $\overline{r}_1, \dots, \overline{r}_m$ then this kernel is $\bigcap_{j=1}^m {\rm ann}(\overline{r}_j)$ where ${\rm ann}(\overline{r}_j)$ is the left sub-ideal of $(P_1\oplus \dots \oplus P_k,-)$ whose value at $Q\in {\mathcal R}$ is the group of morphisms from $P_1\oplus \dots \oplus P_k$ to $Q$ which, precomposed with $H$, give $0$.  Since ${\mathcal R}$ is right coherent that left ideal is finitely generated (the algebraic proof of that, in terms of modules, given in \cite{PreNBK} works just as well here, in view of \ref{cohringoid}).  Therefore the projection to the first $n$ coordinates, that is $\phi({\mathcal R}_{\mathcal R})$, is finitely generated.

If ${\mathcal R}$ is additive then the one-variable condition is a special case of this and, in the preadditive case can be proved just as above or derived from the additive case.

Since annihilators of elements (as seen already) and intersections of finitely generated left ideals (as easily seen) are pp-definable, coherence of ${\mathcal R}$ follows from the other conditions.
\end{proof}

We were able to follow the proof at \cite[2.3.19]{PreNBK} step-for-step, indeed almost word-for-word and so we will not give proofs here for extensions of results such as \ref{flatchar} and \ref{abspurchar} to ${\mathcal R}$-modules.

The essence of this section is the assertion that what works (in this circle of ideas) for modules over rings works also for modules over preadditive categories.  Some statements for the ring context might not be literally meaningful in the more general context but the details above indicate how to modify such statements so as to produce ones which are both meaningful and correct, as well as how to find their proofs.

\section{The category of pp sorts}\label{secppsort}\marginpar{secppsort}

Small abelian categories arise from the model theory as follows.

Because homomorphisms preserve solution sets of pp formulas, for any pp formula $\phi$ the assignment $M\mapsto \phi(M)$ defines a functor, which we denote $ F_\phi $, from ${\rm Mod}\mbox{-}{\mathcal R}$ to $ {\bf Ab}$.  Because (taking the solution set of) a pp formula commutes with direct limits and since every module is a direct limit of finitely presented modules, this functor is determined by its restriction to the subcategory, $ {\rm mod}\mbox{-}{\mathcal R}$, of finitely presented modules. We will use the same notation, $F_\phi$, for this restriction.   If $ \psi \leq \phi $ then $ F_\psi$ is a subfunctor of $F_\phi  $ and we can form the quotient functor $ F_{\phi /\psi }=F_\phi /F_\psi $, which takes a module $M$ to $\phi(M)/\psi(M)$.

Define the {\bf category of pp sorts} for (the theory of) $ {\mathcal R}$-modules to have, for objects, the pp-pairs (which we typically write in the format $ \phi /\psi $) and, for morphisms, the pp-definable maps between such pairs; that is, if $ \phi (\overline{x})/\psi (\overline{x}) $ and $ \phi '(\overline{y})/\psi '(\overline{y}) $ are pp-pairs, then a pp formula $ \rho (\overline{x},\overline{y}) $ defines a map from the first pair to the second if $ \rho (\overline{x},\overline{y})$ and $\phi (\overline{x})$, respectively $ \rho (\overline{x},\overline{y})$ and $ \psi (\overline{x})$, together imply $ \phi '(\overline{y})$, respectively $ \psi '(\overline{y}) $ (clearly this is what is required to have a well-defined map between the corresponding quotients). This category is denoted $ {\mathbb L}^{\rm eq+}_{{\mathcal R}}$. It is a small abelian category and, {\it via} its action evaluation-at-a-finitely-presented-module, it embeds as a subcategory of the category $({\rm mod}\mbox{-}{\mathcal R},{\bf Ab})$ of additive functors from finitely presented ${\mathcal R}$-modules to ${\bf Ab}$.  In fact (\cite[3.2.5]{BurThes}; see, e.g., \cite[10.2.30]{PreNBK}) this embedding is an equivalence with $({\rm mod}\mbox{-}{\mathcal R},{\bf Ab})^{\rm fp}$, the category of finitely presented functors and hence also with the free abelian category on ${\mathcal R}^{\rm op}$ (see e.g.~\cite[\S 10.2.7]{PreNBK}).

\begin{theorem}\label{equiv3}\marginpar{equiv3} (see, e.g., \cite{PreADC}) For any skeletally small preadditive category ${\mathcal R}$ the categories $ ({\rm mod}\mbox{-}{\mathcal R},{\bf Ab})^{\rm fp}$, ${\mathbb L}^{\rm eq+}_{{\mathcal R}}$, ${\rm Ab}({\mathcal R})^{\rm op}$ and ${\rm fun}\mbox{-}{\mathcal R} = ({\rm Mod}\mbox{-}{\mathcal R}, {\bf Ab})^{\rightarrow \prod}$, the category of additive functors which commute with direct products and direct limits, are equivalent.
\end{theorem}

We have, by localisation/relativisation, the same - equivalence between the functor category $({\mathcal D}, {\bf Ab})^{\rightarrow \prod}$, the category $ {\mathbb L}^{\rm eq+}({\mathcal D}) $ of pp sorts and a localisation of a free abelian category - for any definable category $ {\mathcal D}$.  Here, by $ {\mathbb L}^{\rm eq+}({\mathcal D}) $ we denote the category of pp-pairs for $ {\mathcal D} $ - defined just as above but with $ {\mathcal D} $ replacing $ {\rm Mod}\mbox{-}{\mathcal R}$ (and using pp formulas from the model-theoretic language for ${\mathcal R}$-modules where ${\mathcal D}$ is a definable subcategory of ${\rm Mod}\mbox{-}{\mathcal R}$).  To be more precise, the objects of $ {\mathbb L}^{\rm eq+}({\mathcal D}) $ are the pairs of those pp formulas $ \phi /\psi  $ with $ \psi \rightarrow \phi $ in ${\mathcal D}$ (though it is also possible just to use the same objects as ${\mathbb L}^{\rm eq+}_{{\mathcal R}}$ since localisation will introduce new isomorphisms). The morphisms from sort $ \phi /\psi  $ to $ \phi '/\psi ' $ are the equivalence classes, modulo the theory of $ {\mathcal D}$, of pp formulas which define a relation from the group defined by $ \phi /\psi  $ to that defined by $ \phi '/\psi ' $ which is functional in every $ D\in {\mathcal D}$.  If we denote by ${\mathcal L}^{\rm eq+}({\mathcal D})$ the language based on this collection of sorts and arrows between them, then each object $ D\in {\mathcal D} $ can be regarded as/expanded to an $ {\mathcal L}^{\rm eq+}({\mathcal D})$-structure in the obvious way. A fruitful way to regard this expansion is to regard it as the additive functor evaluation-at-$D$, $ {\rm ev}_D$, from $ {\mathbb L}^{\rm eq+}({\mathcal D}) $ to $ {\bf Ab} $ (which takes a sort $ \phi /\psi  $ to the abelian group $ \phi (D)/\psi (D) $ and takes a function symbol to the actual function that it defines from $ \phi (D)/\psi (D) $ to $ \phi '(D)/\psi '(D)$).
Then we may say part of \cite[12.3.20]{PreNBK} in model-theoretic terms as follows (Serre subcategories and localisation/quotient categories are discussed in the same reference, more briefly in \cite{PreADC} and in many other sources).

\begin{theorem}\label{ppsortquotpp}\marginpar{ppsortquotpp} If $ {\mathcal D} $ is a definable subcategory of $ {\rm Mod}\mbox{-}{\mathcal R} $ and $ {\mathcal S}_{{\mathcal D}} $ denotes the full subcategory of $ {\mathbb L}^{\rm eq+}_{\mathcal R}$ on those pp-pairs which are closed on ${\mathcal D}$ then the quotient category $  {\mathbb L}^{\rm eq+}_{\mathcal R}/{\mathcal S}_{{\mathcal D}} $ is, in a natural way, naturally equivalent to $ {\mathbb L}^{\rm eq+}({\mathcal D})$.
\end{theorem}

The naturality of the equivalence is with respect to the actions on $ {\mathcal D}$: if $ \phi/\psi $ is a pp-pair for ${\mathcal R}$-modules then the action of its image in the quotient category is given by the same pp-pair but now restricted to ${\mathcal D}$.

Every small abelian category arises thus, as the category of pp-pairs for some definable category.

\subsection{An example}\label{secexppsort}\marginpar{secexppsort}

Much about the category ${\mathbb L}^{\rm eq+}({\mathcal D})$ is explicitly computable (if ${\mathcal D}$ is given explicitly enough).  In some cases the whole category of pp sorts can be described completely.

\begin{example}\label{kepsilon}\marginpar{kepsilon} \cite[\S 6.8]{Perera} Let $ k $ be a field and let $ R=k[\epsilon ] =k[X]/\langle X^2\rangle $. It is easy to see that every $R$-module has the form $ R^{(\kappa )}\oplus S_1^{(\lambda )} $ for some cardinals $ \kappa ,\lambda  $ where $S_1$ is the unique simple module $R/\langle \epsilon \rangle$.  In particular $ R $ is of {\bf finite representation type}, meaning that every $ R$-module is a direct sum of indecomposable modules and there are, up to isomorphism, only finitely many indecomposable modules. The Auslander-Reiten quiver of $ R $ is easily computed (in this example we assume some knowledge of Auslander-Reiten theory, for which see e.g.~\cite{ARS}) and has arrows the embedding $ i:S_1\rightarrow R $ and the epimorphism $ \pi :R\rightarrow S_1 $. Note that $ \pi i=0 $ and $ i\pi =\epsilon  $ (i.e.~multiplication by $ \epsilon $).

The {\bf Auslander algebra} of $ R $ is the endomorphism ring $  S={\rm End}(M) $ where $ M $ is a direct sum of one copy of each of the indecomposable (finitely presented) $R$-modules, which in this case is $ S={\rm End}(R\oplus S_1) $.  We may represent $S$ as the matrix ring
$\left( \begin{array}{cc} (R,R) &  (S_1,R) \\ (R,S_1) & (S_1,S_1) \end{array} \right) \simeq \left( \begin{array}{cc} R=k1_R\oplus k\epsilon & ki \\ k\pi & k1_{S_1} \end{array} \right)$
and this decomposes as a left module as, say, $ Q_1\oplus Q_2 $ where $ Q_i $ is the $ i$-th column. Set $ T_1=Q_1/{\rm rad}(Q_1) $.  One can check that $ Q_2\simeq {\rm rad}(Q_1)$. Right multiplication by $\left( \begin{array}{cc} 0 & 0 \\ \pi & 0 \end{array} \right)$
gives an epimorphism $ Q_1\rightarrow {\rm rad}(Q_2) $ so, noting the dimensions of these modules, we conclude $ {\rm rad}(Q_2)\simeq T_1 $.  Let $ T_2 $ denote the other simple module (the top of $ Q_2$).

Thus far we have four indecomposable modules - the two indecomposable projectives and the two simples. One can also see $ I_2 $ - the injective hull of $ S_2$, with socle $ S_2 $ and $ I_2/S_2\simeq S_1$. It is not difficult to check that there are no more indecomposable modules so, since $ S $ is an artin algebra, $ S $ is of finite representation type and we can compute its Auslander-Reiten quiver (for left modules) to be as shown, where the 1st (respectively 2nd) and 5th (respectively 6th) columns should be identified (and dotted lines indicate Auslander-Reiten translates).

$
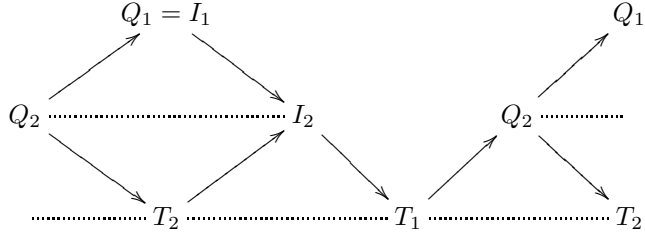
\xymatrix{ & Q_1=I_1 \ar[rd] & & & & Q_1 \\ Q_2 \ar[ru] \ar[rd] \ar@{.}[rr] & & I_2 \ar[rd] & & Q_2 \ar[rd] \ar[ru] \ar@{.}[r] & \\ \ar@{.}[r] & T_2 \ar[ur] \ar@{.}[rr] & & T_1 \ar[ur] \ar@{.}[rr] & & T_2 }$

Now we use \ref{equiv3} and the fact (see, e.g., \cite[4.9.4]{BenBk1} that the category ${\mathbb L}^{\rm eq+}_R$ of pp-sorts (in the form of the category of finitely presented functors on finitely presented right $R$-modules) is actually equivalent to the category of left modules over the Auslander algebra: $ S\mbox{-}{\rm mod}\simeq {\mathbb L}^{\rm eq+}_R$.  We deduce that, in our example, there are just five indecomposable pp-sorts (we remark that in general the Auslander algebra of an algebra of finite representation type need not itself be of finite representation type).  To identify these sorts we can use the explicit description (e.g., see \cite[p.~121]{BenBk1}) of this equivalence, which takes a pp-pair $\phi/\psi$ to $\phi(M)/\psi(M)$ regarded as a left $S$-module and, in the other direction, a finitely presented left $S$-module $_S A$ is sent to the functor ${\rm Hom}_S(_S{\rm Hom}_R(-,\,_SM_R), \,_SA)$, which can be rewritten, given a presentation of $A$, in terms of pp formulas.  In our example all this is easily computed:

\noindent $ x=x $ (meaning the sort $ (x=x)/(x=0)$) corresponds to the $ S$-module $ Q_1$;

\noindent $ x\epsilon =0 $ to $Q_2$;

\noindent $ \epsilon \mid x $ (meaning $ (\exists y \,\, x=y\epsilon ) / (x=0)$) corresponds to $ T_1$;

\noindent $ (x=x)/(\epsilon \mid x )$ to $I_2$;

\noindent $ (x\epsilon =0)/(\epsilon \mid x )$ to $T_2$.

If we evaluate each of these in turn on $ R\oplus S_1 $ then we obtain the vector spaces $ R\oplus S_1$; $ \epsilon k\oplus S_1$; $ \epsilon k\oplus 0$; $ R/\epsilon k\oplus S_1$; $0\oplus S_1$, each equipped with the natural $ S$-action.

To illlustrate the localisation process of \ref{ppsortquotpp} (passing to a definable subcategory, mirrored by localisation of the category of pp-sorts), let $ {\mathcal D} $ be the definable subcategory of $ {\rm Mod}\mbox{-}R $ generated by $ R_R$. The annihilator $ {\mathcal S}_{{\mathcal D}} $ of $ {\mathcal D} $ in $ {\mathbb L}^{\rm eq+}_R $ is, checking the list above, the Serre subcategory generated by $ (x\epsilon =0)/(\epsilon \mid x) $, so we have $ {\mathbb L}^{\rm eq+}({\mathcal D})\simeq {\mathbb L}^{\rm eq+}_R/\langle (x\epsilon =0)/(\epsilon \mid x)\rangle  \simeq  S\mbox{-}{\rm mod}/\langle T_2\rangle $, where $\langle - \rangle$ denotes the Serre subcategory generated by ($-$). Therefore, in the quotient category, since $ T_2 $ becomes isomorphic to $ 0$, all of $ Q_2$, $ T_1 $ and $ I_2 $ become isomorphic and we see that there are just two indecomposables, a simple and a non-split self-extension of that simple - and we recognise that the quotient category looks very much like $ {\rm mod}\mbox{-}R$, indeed \ref{cohexgen} below tells us that it is equivalent to $ ({\rm mod}\mbox{-}R)^{\rm op} $ which, for this particular (commutative) ring, is indeed equivalent to $ {\rm mod}\mbox{-}R$.
\end{example}

\section{Elimination of pp-imaginaries}\label{secelimimag}\marginpar{secelimimag}

It has been known for a long time that the theory of modules over a ring $R$ has {\bf elimination of quantifiers} (every formula is equivalent to one without quantifiers) iff $R$ is {\bf von Neumann regular} (for every $ r\in R $ there is $ s\in R $ with $ rsr=r$, equivalently, every finitely generated (say) right ideal is a direct summand of $R$), see \cite[16.16, 16.17]{PreBk} and references given there.  Elimination of quantifiers in general implies only that a pp formula is equivalent to a quantifier-free formula but for a definable additive category ${\mathcal D}$ its theory has elimination of quantifiers (in a particular language) iff every pp formula is equivalent to a conjunction of atomic formulas; the proof at \cite[16.5(a)]{PreBk} for definable subcategories of the category of modules over a ring works in the general case.  First we show that there is a similar reduction for pp-imaginaries.  We do assume more acquaintance with basic model theory in this section.

Elimination of quantifiers or of imaginaries is always with respect to some language ${\mathcal L}$.  Moreover, for the latter we need also to specify a set, ${\mathcal H}$, of ``home sorts" - that is, a set of objects of the category of pp sorts - which we may as well be assume to be closed under finite products.  A theory $T$ in ${\mathcal L}$ has {\bf elimination of pp-imaginaries to} ${\mathcal H}$ if every pp-definable sort is in definable bijection with a definable subset of some sort in ${\mathcal H}$.  We show that the definable bijection in this definition may be taken to be pp-definable provided that the class of models of $T$ is closed under (finite) direct sums.

\begin{prop}\label{elimimag+}\marginpar{elimimag+} Suppose that ${\mathcal D}$ is a definable additive category, definable by closure of a set of pp-pairs in the language ${\mathcal L}$ and let ${\mathcal H}$ be an additive subcategory of ${\mathbb L}^{\rm eq+}({\mathcal D})$ (where the latter is defined with respect to ${\mathcal L}$).  If the ${\mathcal L}$-theory of ${\mathcal D}$ has elimination of pp-imaginaries to ${\mathcal H}$ then every object of ${\mathbb L}^{\rm eq+}({\mathcal D})$ is pp-definably isomorphic to a subobject of some object of ${\mathcal H}$.
\end{prop}
\begin{proof}  We use a kind of argument also seen, e.g., in \cite[proof of 6.3.5]{Perera}.  Suppose that $\phi/\psi$ is a pp sort and that there is a definable embedding $\phi/\psi \rightarrow H$ with $H\in {\mathcal H}$, defined by the formula $\chi(x,y)$ say, where $x$ is a variable of sort $\phi/\psi$ (or of sort $\phi$ if one prefers) and $y$ is of sort $H$.  Then $\chi$ has the form $\bigvee_i \chi_i$ where each $\chi_i$ is a conjunction of atomic and negated atomic formulas.  If $\chi$ were not equivalent (modulo the theory of ${\mathcal D}$) to any $\chi_i$ then we choose, for each $i$, some $D_i\in {\mathcal D}$ and $(a_i,b_i)$ from $D_i$ with $D_i\models \chi (a_i,b_i)\wedge \neg \chi_i(a_i,b_i)$.  Then $\bigoplus_i D_i$ is in ${\mathcal D}$ but contains the tuple $((a_i)_i,(b_i)_i)$ which satisfies $\chi$ but no $\chi_i$, yet lies in sort $\phi/\psi$, contradicting that $\chi$ defines a total function.

So $\chi$ has the form $\theta \wedge \bigwedge_j \neg\theta_j$ where $\theta$ is a conjunction of atomic formulas and the $\theta_j$ are atomic formulas.  We may assume that $\exists y \chi(x,y)$ is equivalent to $(\phi/\psi)(x)$.  Suppose that $D\in {\mathcal D}$ and $D\models \chi(a,b)$.  Then $\theta(a,D) =\{ b\} \cup \bigcup_j \theta_j(a,D)$.  All these are cosets of subgroups (of $H(D)$) so, by Neumann's Lemma (see, e.g., \cite[2.12]{PreBk}), and since all indices of one definable subgroup in another are infinite if not equal to 1, and since $b\notin \theta_j(a,D)$ for any $j$, we deduce that $\theta(a,D) =\{ b\}$.  That is, $\theta$ already defines a function from $\phi/\psi$ to $H$ which, since we have just seen that for each $a$ of sort $\phi/\psi$, $\theta(a,D)$ is a singleton and since $\chi(a,b)$ implies $\theta(a,b)$, must be injective, as required.
\end{proof}

It was shown at \cite[10.2.40]{PreNBK} that the theory of modules over a ring $R$ has {\bf elimination of pp-imaginaries}, meaning elimination of pp-imaginaries to the additive category with objects the finite powers $(R_R,-)^n$ of the home sort, iff $R$ is von Neumann regular.  For a ring, that is the relevant set of ``home sorts"; extension to a skeletally small preadditive category ${\mathcal R}$ is straightforward.  The elementwise definition of von Neumann regularity (where now $r$ and $s$ are to be read as morphisms of ${\mathcal R}$) can be used and is equivalent to the condition that every finitely generated ideal contained in a representable functor $(P,-)$ ($P\in {\mathcal R}$) be a direct summand.  The proofs, that both pp-elimination of quantifiers and elimination of pp-imaginaries are equivalent to von Neumann regularity, that are given in \cite{PreNBK} apply equally well to ${\mathcal R}$ as to $R$.  (Elimination of quantifiers is \cite[10.2.38]{PreNBK}; it is not said explicitly there but follows by \cite[10.2.14(c)]{PreNBK} and \cite[16.5]{PreBk}.)  So we have the following.

\begin{theorem}\label{elimelim}\marginpar{elimelim} For a skeletally small preadditive category ${\mathcal R}$, the following are equivalent:

\noindent (i) ${\mathcal R}$ is von Neumann regular;

\noindent (ii) the theory of ${\mathcal R}$-modules (right, equivalently left) has elimination of quantifiers;

\noindent (iii) the theory of ${\mathcal R}$-modules (right, equivalently left) has pp-elimination of quantifiers;

\noindent (iv) the theory of ${\mathcal R}$-modules has elimination of pp-imaginaries.
\end{theorem}

The underlying reason that elimination of quantifiers and elimination of pp-imaginaries are equivalent would seem to be the characterisation, see \cite[10.2.14]{PreNBK}, due to Burke, of subsorts of a power $(R,-)^n$ of the home sort as those objects of ${\mathbb L}_R^{\rm eq+}$ of projective dimension $\leq 1$, and of those defined by conjunctions of atomic formulas as being those of projective dimension $0$ (that is, projective, equivalently a representable functor $(A,-)$ for $A$ some finitely presented ${\mathcal R}$-module).  Combined with Auslander's result (\cite[p.~205]{AusCoh}) that if there is a non-projective object of ${\mathbb L}_R^{\rm eq+}$ then there is an object of projective dimension $=2$ (the maximum possible for objects of these categories), this shows that elimination of pp-imaginaries, that is, every object having projective dimension $\leq 1$, already implies elimination of quantifiers.

It is not clear what can be said of more general definable categories ${\mathcal D}$ because there seems to be no clear candidate for a canonical choice of set of home sorts.  Of course one may choose ${\mathcal H}$ to contain all pp sorts, in which case we have both eliminations but they are devoid of content:  we may regard ${\mathcal D}$ as a definable subcategory of ${\mathbb L}^{\rm eq+}({\mathcal D})\mbox{-}{\rm Mod}$; then every pp-sort for the language of ${\mathbb L}^{\rm eq+}({\mathcal D})$-modules is, when restricted to ${\mathcal D}$, isomorphic to a representable=projective sort (that is, applying the $^{\rm eq+}$ construction twice is equivalent to doing it once).  In the case that ${\mathcal D}$ is a category of modules this issue is resolved by the fact we have just seen, that the sorts of projective dimension $\leq 1$ are those which are the natural home sorts whenever we present the category as a module category (projective dimension, being a categorical invariant, is preserved under any representation of ${\mathcal D}$ as a definable subcategory).  For a general definable category ${\mathcal D}$, the category ${\mathbb L}^{\rm eq+}({\mathcal D})$ of sorts can be any small abelian category and so, for instance, there may be no projective objects at all - the category, ${\rm fin}\mbox{-}{\mathbb Z}$, of finite abelian groups is an example of a small abelian category with no non-zero projectives.  Let us compute the corresponding definable category which, abstractly, is ${\mathcal D} ={\rm Ex}({\rm fin}\mbox{-}{\mathbb Z},{\bf Ab})$ (if ${\mathcal A}$ is a skeletally small abelian category then the definable category for which it is the category of pp-imaginaries is the category, ${\rm Ex}({\mathcal A}, {\bf Ab})$, of exact functors from ${\mathcal A}$ to ${\bf Ab}$, see, say \cite{PreADC}).

\begin{example}\label{finZ}\marginpar{finZ}.  We compute ${\mathcal D} ={\rm Ex}({\rm fin}\mbox{-}{\mathbb Z},{\bf Ab})$.  In fact, it seems a little easier to compute the elementary dual (see the next section), ${\mathcal D}^{\rm d} ={\rm Ex}(({\rm fin}\mbox{-}{\mathbb Z})^{\rm op},{\bf Ab})$.  We have the inclusion $i:({\rm fin}\mbox{-}{\mathbb Z})^{\rm op} \rightarrow ({\rm mod}\mbox{-}{\mathbb Z})^{\rm op}$ and the definable category corresponding to the latter is, by \ref{cohexgendual} below, ${\rm Abs}\mbox{-}{\mathbb Z} = {\rm Inj}\mbox{-}{\mathbb Z}$ - the category of injective=divisible abelian groups.  This (exact) inclusion of abelian categories corresponds, see e.g.~\cite{PreRajShv} or \cite{PreADC}, to an interpretation functor ${\rm Inj}\mbox{-}{\mathbb Z} \rightarrow {\mathcal D}^{\rm d}$ which, since $({\rm fin}\mbox{-}{\mathbb Z})^{\rm op}$ is a Serre subcategory of $({\rm mod}\mbox{-}{\mathbb Z})^{\rm op}$ is a definable quotient category of ${\rm Inj}\mbox{-}{\mathbb Z}$ in the sense of \cite[\S 5]{KraEx}.  The definition is as follows.  Suppose that ${\mathcal C}$ is a definable category and that ${\mathcal S}$ is a Serre subcategory of ${\rm fun}({\mathcal C}) =({\mathcal C}, {\bf Ab})^{\rightarrow \prod}$; then ${\mathcal S}$ is, in particular, an abelian category and the corresponding definable category ${\mathcal C}'' ={\rm Ex}({\mathcal S}, {\bf Ab})$ is the corresponding {\bf definable quotient category} of ${\mathcal C}$.  Krause shows the following where we write ${\rm Pinj}({\mathcal D}$ for the full subcategory on the pure-injective ($=$ pp-saturated) objects of ${\mathcal D}$.

\begin{theorem}\label{defquot}\marginpar{defquot} (\cite[5.1]{KraEx}) Suppose that ${\mathcal A}$ is a small abelian category, that ${\mathcal S}$ is a Serre subcategory, that ${\mathcal C} ={\rm Ex}({\mathcal A}, {\bf Ab})$, that ${\mathcal C}'' ={\rm Ex}({\mathcal S}, {\bf Ab})$, and that ${\mathcal C}' ={\rm Ex}({\mathcal A}_{\mathcal S}, {\bf Ab})$, where ${\mathcal A}_{\mathcal S}$ denotes the localisation of ${\mathcal A}$ at ${\mathcal S}$ (so ${\mathcal C}'$ is a definable subcategory of ${\mathcal C}$ and ${\mathcal C}''$ is the corresponding definable quotient category of ${\mathcal C}$).

Then ${\rm Pinj}({\mathcal C}'') \simeq {\rm Pinj}({\mathcal C})/{\rm Pinj}({\mathcal C}')$ where the latter is the stable quotient category of ${\rm Pinj}({\mathcal C})$ with respect to ${\rm Pinj}({\mathcal C}')$ - that is, the objects are those of ${\rm Pinj}({\mathcal C})$ and the morphisms are those of ${\rm Pinj}({\mathcal C})$ modulo those which factor through an object of ${\rm Pinj}({\mathcal C}')$.
\end{theorem}

Clearly, in our case, the definable quotient category of ${\rm Inj}\mbox{-}{\mathbb Z}$ corresponding to ${\mathcal D}^{\rm d}$ consists of the torsionfree divisible groups.  We have chosen to work with the dual/opposite categories exactly because every object of ${\rm Inj}\mbox{-}{\mathbb Z}$ is pure-injective and so this theorem gives us a description of the whole category ${\mathcal D}^{\rm d}$, namely it is ${\rm Inj}\mbox{-}{\mathbb Z}/{\rm Mod}\mbox{-}{\mathbb Q}$ and that, since there are no non-zero morphisms from torsion divisible injectives to torsionfree ones, is the category of torsion divisible abelian groups.

We conclude that ${\rm Ex}(({\rm fin}\mbox{-}{\mathbb Z})^{\rm op},{\bf Ab})$ is the category of torsion divisible abelian groups and, by elementary duality, that ${\rm Ex}({\rm fin}\mbox{-}{\mathbb Z},{\bf Ab})$ is the category of reduced (without nonzero divisible submodules) flat abelian groups.
\end{example}

Finally in this section, note that the tensor embedding, $M\mapsto M\otimes_R(-)$, of ${\rm Mod}\mbox{-}{\mathcal R}$ into $({\mathcal R}\mbox{-}{\rm mod})\mbox{-}{\rm Mod} = ({\mathcal R}\mbox{-}{\rm mod}, {\bf Ab})$ already eliminates pp-imaginaries and quantifiers for the theory of $ {\rm Mod}\mbox{-}{\mathcal R}$ (\cite[11.1.44]{PreNBK}) - it is not necessary to go to the larger category ${\mathbb L}^{\rm eq+}_{\mathcal R}\mbox{-}{\rm Mod}$.  Of course, there is no reason why the theory of (${\mathcal R}\mbox{-}{\rm mod}$)-modules should have these eliminations.  But we can repeat the process, replacing ${\mathcal R}$ by ${\mathcal R}\mbox{-}{\rm mod}$, and then repeating again, through the natural numbers.  This gives a directed system of embeddings of module categories, and hence also of languages, and the direct limit will be a module category (the embeddings preserve the, generating at each stage, projective objects) and will have elimination of quantifiers and pp-imaginaries since every sort is definable using just finitely many symbols and so already appears, and embeds in a projective sort, at some stage.

\section{Duality}\label{secdual}\marginpar{secdual}

Recall that if $ {\mathcal D}={\rm Ex}({\mathcal A},{\bf Ab}) $ is a definable category then we denote its (elementary) dual ${\rm Ex}({\mathcal A}^{\rm op},{\bf Ab}) $ by $ {\mathcal D}^{\rm d}$ and (see \cite[10.3.4]{PreNBK}) for such a dual pair the categories of pp pairs are opposite to each other: $ {\mathbb L}^{\rm eq+}({\mathcal D}^{\rm d}) \simeq {\mathcal A}^{\rm op} \simeq ({\mathbb L}^{\rm eq+}({\mathcal D}))^{\rm op}$.  This duality works at the level of pp formulas and pp-pairs as follows.

If $R$ is a ring and if $ \phi (\overline{x})$, where $\overline{x} =(x_1,\dots, x_n) $, is the condition $
\exists \overline{y}\, (\overline{x}\, \overline{y})H=0 $ for right modules, where $H= \left( \begin{array}{c} A \\ B \end{array}
\right)$ with the matrix $A$ having $n$ rows, then the
{\bf dual} condition, $ D\phi(\overline{x}) $ (where now $ \overline{x} $ is a column vector of length $n$ and
will be substituted by elements from {\em left} modules) is defined to
be  $\exists \overline{z} \left( \begin{array}{cc} I & A \\ 0 & B
\end{array} \right) \left( \begin{array}{c} \overline{x}\\ \overline{z}
\end{array} \right) =0 $.  For the description of the duality from left to right, also denoted $D$, we just transpose everything.  This is a duality in the sense that $DD\phi$ is equivalent to $\phi$ and $\psi\leq \phi$ iff $D\phi \leq D\psi$ (\cite{PreDual} or \cite[\S 1.3]{PreNBK}).

The description in the case of rings ${\mathcal R}$ with many objects is formally the same but in the details one has to pay attention to the sorting of symbols.  Let us again use subscripts to indicate the sorts of variables and also, if $r:P\rightarrow Q$ is a morphism of ${\mathcal R}$, we write $r_{QP}$ for the corresponding function symbol for right modules (noting that it has domain sort $Q$ and codomain sort $P$) and $r_{PQ}$ for the corresponding symbol in the language for left ${\mathcal R}$-modules.  Then a pp formula for right modules has the form $$\exists y_{Q_1},\dots ,y_{Q_m} \,\, (x_{P_1},\dots, x_{P_n}, y_{Q_1},\dots ,y_{Q_m}) \left( \begin{array}{ccccc} r_{P_1R_1} & \dots & r_{P_1R_j} & \dots & r_{P_1R_k} \\ \vdots & \ddots & \vdots & \ddots & \vdots \\  r_{P_nR_1} & \dots & r_{P_nR_j} & \dots & r_{P_nR_k} \\ s_{Q_1R_1} & \dots & s_{Q_1R_j} & \dots & s_{Q_1R_k} \\ \vdots & \ddots & \vdots & \ddots & \vdots \\ s_{Q_mR_1} & \dots & s_{Q_mR_j} & \dots & s_{Q_mR_k} \end{array} \right) = (0_{R_1} ,\dots, 0_{R_k}) .$$  The dual of this formula is then  $$\exists z_{R_1},\dots ,z_{R_k} \,\, \left( \begin{array}{cccccccc} 1_{P_1} & \dots & 0 & r_{R_1P_1} & \dots & r_{R_jP_1} & \dots & r_{R_kP_1} \\ \vdots & \ddots & \vdots & \vdots & \ddots & \vdots & \ddots & \vdots \\  0 & \dots & 1_{P_n} & r_{R_1P_n} & \dots & r_{R_jP_n} & \dots & r_{R_kP_n} \\ 0 & \dots & 0 & s_{R_1Q_1} & \dots & s_{R_jQ_1} & \dots & s_{R_kQ_1} \\ \vdots & \ddots & \vdots & \vdots & \ddots & \vdots & \ddots & \vdots \\  0 & \dots & 0 & s_{R_1Q_m} & \dots & s_{R_jQ_m} & \dots & s_{R_kQ_m} \end{array} \right) \left(\begin{array}{c} x_{P_1} \\ \vdots \\ x_{P_n} \\ z_{R_1} \\ \vdots \\ z_{R_k} \end{array}\right) = \left( \begin{array}{c} 0_{P_1} \\ \vdots \\ 0_{P_n} \\ 0_{Q_1} \\ \vdots \\ 0_{Q_m} \end{array} \right) .$$

This duality extends to pp-pairs, taking the pair $\phi/\psi$ to $D\psi/D\phi$, as well as to pp-definable functions between them, reversing their direction (\cite{HerzDual}, see \cite[p.~94ff.]{PreNBK}).  This, in turn, induces the duality bijection between definable subcategories:  if the definable subcategory ${\mathcal D}$ of ${\rm Mod}\mbox{-}{\mathcal R}$ is axiomatised by closure of the set $\{ \phi_\lambda /\psi_\lambda\}_\lambda$ of pp-pairs then its elementary dual ${\mathcal D}^{\rm d}$ is the definable subcategory of ${\mathcal R}\mbox{-}{\rm Mod}$ defined by the set $\{ D\psi_\lambda /D\phi_\lambda\}_\lambda$ of pp-pairs (by what has been said already, this is independent of representation of ${\mathcal D}$ as a definable subcategory).

Many of the results around the (model) theory of flat and absolutely pure modules can be given quick proofs using the following result of Herzog (\cite[3.2]{HerzDual} for rings but the proof is easily modified to the general case).

First, we recall the definition of tensor product of right and left modules over a small preadditive category ${\mathcal R}$ (see, e.g., \cite[pp.~493,494]{PreNBK} for more detail and a worked-through example).  If $M$ is a right ${\mathcal R}$-module then, if $Q$ is an object of ${\mathcal R}$, so $(Q,-)$ is a left ${\mathcal R}$-module then $M\otimes_{\mathcal R}(Q,-)$ is defined to be $MQ$.  Adding the condition that $M\otimes_{\mathcal R}-$ is a right exact functor is then sufficient additional specification because every left module $L$ has a presentation $L_1\rightarrow L_0 \rightarrow L \rightarrow 0$ where $L_1$ and $L_0$ are direct sums of modules of the form $(Q,-)$.

Consider now $(-,P)\otimes_{\mathcal R}(Q,-)$ where $P, Q$ are objects of ${\mathcal R}$.  An element of $(-,P)$ of sort $P'$ is a morphism $P'\xrightarrow{r}P$ and an element of $(Q,-)$ of sort $Q'$ is a morphism $Q\xrightarrow{s}Q'$.  By definition $(-,P)\otimes_{\mathcal R}(Q,-) = (Q,P)$ and so we see that, in order to ``tensor" or compose (to give a morphism $Q\xrightarrow{rs}P$) we must have $P'=Q'$, that is, for ``$r\otimes s$" to make sense, $r$ and $s$ must have the `same' sort.

\begin{theorem}\label{herzcrit}\marginpar{herzcrit}  (Herzog's
Criterion) Suppose that ${\mathcal R}$ is a skeletally small preadditive category.
 Let $\overline{r}$ from $M\in {\rm Mod}\mbox{-}{\mathcal R}$ and $\overline{s}$ from
$N \in {\mathcal R}\mbox{-}{\rm Mod}$ be tuples of the same length and matching sorts (that is, if $r_i$ has sort $P'$ then $s_i$ has sort $P'$).  Then
 $\overline{r}\otimes \overline{s}=0$ in $M\otimes_{\mathcal R} N$ iff there is a
pp condition $\phi$ (for right modules) such that $\overline{r} \in
 \phi(M)$ and $\overline{s}\in D\phi (N)$.
\end{theorem}

\section{${\mathcal A}({\mathcal R})$, flat and absolutely pure}\label{secARflat}\marginpar{secARflat}

We describe the abelian category ${\mathcal A}({\mathcal R})$, which will be a quotient of the category, ${\mathbb L}^{\rm eq+}_{{\mathcal R}^{\rm op}}$, of pp-imaginaries, first category-theoretically, then in terms of pp formulas.

As always, ${\mathcal R}$ denotes a skeletally small preadditive category.  We have the standard embedding of ${\mathcal R}$ into its category of right ${\mathcal R}$-modules which takes $P \in {\mathcal R}$ to the representable functor/projective right ${\mathcal R}$-module/right ideal of ${\mathcal R}$ $(-,P)$ and a morphism $r:P\rightarrow Q$ to $(-,r):(-,P) \rightarrow (-,Q)$. By the defining property of the free abelian category (see, e.g., \cite[\S 10.2.7]{PreNBK}) this factors through the free abelian category on ${\mathcal R}$.

$\xymatrix{{\mathcal R} \ar[rr] \ar[dr] & & {\rm Ab}({\mathcal R})=({\mathcal R}\mbox{-}{\rm mod},{\bf Ab})^{\rm fp} ={\mathbb L}^{\rm eq+}_{{\mathcal R}^{\rm op}} \ar@{.>}[dl] \\ & {\rm Mod}\mbox{-}{\mathcal R}}$

The functor from $ {\rm Ab}({\mathcal R}) $ to $ {\rm Mod}\mbox{-}{\mathcal R} $ is ``evaluation at $ {\mathcal R}$": in the case of a ring $ R $ it is $ F\mapsto F(_R R) $ and in the general case it takes $ F $ to the functor from $ {\mathcal R} $ to $ {\bf Ab} $ which, at an object $ P $ of $ {\mathcal R} $, has value $ F(P,-) $.  We define $ {\mathcal A}({\mathcal R}) $ to be the image of $ {\rm Ab}({\mathcal R}) $ under this functor; it is an, in general non-full, abelian subcategory of $ {\rm Mod}\mbox{-}{\mathcal R} $.  In fact it is, see \cite[\S6]{PreRajShv}, the smallest abelian, not necessarily full, subcategory of $ {\rm Mod}\mbox{-}{\mathcal R} $ which contains the category, $ {\rm mod}\mbox{-}{\mathcal R}$, of finitely presented modules and so it coincides with $ {\rm mod}\mbox{-}{\mathcal R} $ exactly in the case that $ {\mathcal R} $ is right coherent (\ref{cohringoid}).  The kernel of this evaluation functor is the Serre subcategory $ {\mathcal Z}_{{\mathcal R}}=\{ F\in {\rm Ab}({\mathcal R}): F({\mathcal R})=0\}$; therefore $ {\mathcal A}({\mathcal R}) \simeq {\rm Ab}({\mathcal R})/{\mathcal Z}_{\mathcal R}$.

If we think of ${\rm Ab}({\mathcal R})$ as being the category, ${\mathbb L}^{\rm eq+}_{{\mathcal R}^{\rm op}}$, of pp sorts for left ${\mathcal R}$-modules then the functor from ${\mathbb L}^{\rm eq+}_{{\mathcal R}^{\rm op}}$ to ${\rm Mod}\mbox{-}{\mathcal R}$ takes a pp-pair $\phi/\psi$ for left ${\mathcal R}$-modules to the right module ``$\phi({\mathcal R})/\psi({\mathcal R})$", meaning the functor which takes an object $P$ of ${\mathcal R}$ to $\phi(P,-)/\psi(P,-)$.  The kernel, ${\mathcal Z}_{\mathcal R}$, therefore consists of the functors which are $0$ when evaluated on each of the representable functors $(P,-)$, that is, on a generating set of projective left modules.  By the natural bijection, see e.g.~\cite[12.4.1]{PreNBK}, between Serre subcategories of ${\mathbb L}^{\rm eq+}_{{\mathcal R}^{\rm op}}$ and definable subcategories of ${\mathcal R}\mbox{-}{\rm Mod}$ we deduce that ${\mathcal A}({\mathcal R})$ is the functor category of the definable category, $\langle {\mathcal R}\mbox{-}{\rm Proj}\rangle$, generated by the projective left ${\mathcal R}$-modules.  Since definable categories are closed under directed colimits, this definable subcategory always contains ${\mathcal R}\mbox{-}{\rm Flat}$ and so can be written also as $\langle {\mathcal R}\mbox{-}{\rm Flat}\rangle$.  From this we deduce the following.

\begin{prop}\label{cohexgen}\marginpar{cohexgen} If $ {\mathcal R} $ is any skeletally small preadditive category then $${\rm Ex}({\mathcal A}({\mathcal R}),{\bf Ab})\simeq \langle {\mathcal R}\mbox{-}{\rm Flat}\rangle .$$ If $ {\mathcal R} $ is right coherent, so $  {\mathcal A}({\mathcal R})={\rm mod}\mbox{-}{\mathcal R} $ and $ {\mathcal R}\mbox{-}{\rm Flat} $ is a definable subcategory of $ {\mathcal R}\mbox{-}{\rm Mod}$, then $${\rm Ex}({\rm mod}\mbox{-}{\mathcal R},{\bf Ab})\simeq {\mathcal R}\mbox{-}{\rm Flat}.$$
\end{prop}

That is, ${\mathcal A}(R)$ is the category of pp-imaginaries for the theory of the left $R$-module $R$, more generally, for the theory of flat left ${\mathcal R}$-modules.  It consists of those right ${\mathcal R}$-modules which can be written in the form $\phi({\mathcal R})/\psi({\mathcal R})$ (in the meaning we gave to this above) for some pair of pp formulas $\phi/\psi$ for left ${\mathcal R}$-modules.  These modules already appear in \cite{GarkGendual}, \cite{Garkrelhom} and \cite{WilkThes}.  It is shown as \cite[6.4]{PreRajShv} that the objects of ${\mathcal A}({\mathcal R})$ are exactly the kernels of morphisms between finitely presented modules.  We note that this class (which in general can include infinitely generated modules) includes every finitely generated submodule of a finitely presented module since if we have an exact sequence $ 0\rightarrow B'\rightarrow B\rightarrow B''\rightarrow 0 $ with both $ B $ and $ B'' $ in $ {\rm mod}\mbox{-}{\mathcal R} $ then, by \cite[6.4]{PreRajShv} the kernel, $ B'$, is in $ {\mathcal A}({\mathcal R})$.

Since evaluation is exact, one sees that, given a right module $M\in {\mathcal A}({\mathcal R})$, the corresponding sort for a left module $L \in \langle {\mathcal R}\mbox{-}{\rm Flat}\rangle$ is $M\otimes_{\mathcal R}L=0$.

Now, we turn to the opposite category, $ {\mathcal A}({\mathcal R})^{\rm op} $, which also is abelian and is the pp-imaginaries category for the definable category elementary-dual to $ \langle {\mathcal R}\rangle $. That, by \cite[4.1]{PRZ1} (for the ring case), is $ \langle {\rm Abs}\mbox{-}{\mathcal R}\rangle =\langle {\rm Inj}\mbox{-}{\mathcal R}\rangle $, so we have the following.

\begin{prop}\label{cohexgendual}\marginpar{cohexgendual} If $ {\mathcal R} $ is any skeletally small preadditive category then $${\rm Ex}({\mathcal A}({\mathcal R})^{\rm op},{\bf Ab})\simeq \langle  {\rm Abs}\mbox{-}{\mathcal R}\rangle  .$$ If $ {\mathcal R} $ is right coherent, so $ {\rm Abs}\mbox{-}{\mathcal R}$ is a definable subcategory of $ {\rm Mod}\mbox{-}{\mathcal R}$, then $${\rm Ex}(({\rm mod}\mbox{-}{\mathcal R})^{\rm op},{\bf Ab})\simeq {\rm Abs}\mbox{-}{\mathcal R}.$$
\end{prop}

We finish by observing that in general $ {\mathcal A}({\mathcal R}^{\rm op})\not\simeq {\mathcal A}({\mathcal R})^{\rm op}$, indeed there are rings $ R $ (such as ${\mathbb Z}$) which are right and left coherent such that $ {\rm mod}\mbox{-}R^{\rm op}\not\simeq ({\rm mod}\mbox{-}R)^{\rm op} $.

\end{document}